\documentclass{daj}

\usepackage{amsmath}
\usepackage{amssymb}
\usepackage{amsthm}
\usepackage{bm}
\usepackage{bbm}
\usepackage{hyperref}
\usepackage[capitalize]{cleveref}

\theoremstyle{plain}
\newtheorem{thm}{Theorem}
\newtheorem{lemma}[thm]{Lemma}
\newtheorem{cor}[thm]{Corollary} 
\newtheorem{prop}[thm]{Proposition}
\newtheorem{conj}[thm]{Conjecture}
  
\theoremstyle{definition}
\newtheorem{defn}[thm]{Definition}

\theoremstyle{remark}
\newtheorem*{rem}{Remark}

\numberwithin{thm}{section}
\numberwithin{equation}{section}

\newcommand{\abs}[1]{\left\lvert#1\right\rvert}
\newcommand{\paren}[1]{\left( #1 \right)}
\newcommand{\ol}{\overline}

\newcommand{\C}{\mathbb C}
\newcommand{\E}{\mathbb E}
\newcommand{\F}{\mathbb F}
\newcommand{\R}{\mathbb R}
\newcommand{\Z}{\mathbb Z}

\DeclareMathOperator{\rank}{rank}
\DeclareMathOperator{\arank}{arank}
\DeclareMathOperator{\prank}{prank}
\DeclareMathOperator{\codim}{codim}
\DeclareMathOperator{\id}{id}
\DeclareMathOperator{\Poly}{Poly}
\DeclareMathOperator{\CSM}{CSM}
\DeclareMathOperator{\nCSM}{nCSM}

\dajAUTHORdetails{%
  title = {Quantitative Bounds for the $U^4$-Inverse Theorem over Low Characteristic Finite Fields}, 
  author = {Jonathan Tidor},
  plaintextauthor = {Jonathan Tidor},
  plaintexttitle = {Quantitative Bounds for the U4-Inverse Theorem over Low Characteristic Finite Fields}, 
  runningtitle = {Quantitative $U^4$-Inverse Theorem over Low Characteristic Finite Fields}, 
}   

\dajEDITORdetails{%
   year={2022},
   number={14},
   received={12 October 2021},   
   published={28 October 2022},  
   doi={10.19086/da.38591},       
}   

\begin{document}

\begin{frontmatter}[classification=text]


\author[jtidor]{Jonathan Tidor\thanks{The author was supported by NSF Graduate Research Fellowship Program DGE-1745302.}}

\begin{abstract}
This paper gives the first quantitative bounds for the inverse theorem for the Gowers $U^4$-norm over $\F_p^n$ when $p=2,3$. We build upon earlier work of Gowers and Mili\'cevi\'c who solved the corresponding problem for $p\geq 5$. Our proof has two main steps: \emph{symmetrization} and \emph{integration} of low-characteristic trilinear forms. We are able to solve the integration problem for all $k$-linear forms, but the symmetrization problem we are only able to solve for trilinear forms. We pose several open problems about symmetrization of low-characteristic $k$-linear forms whose resolution, combined with recent work of Gowers and Mili\'cevi\'c, would give quantitative bounds for the inverse theorem for the Gowers $U^{k+1}$-norm over $\F_p^n$ for all $k,p$.
\end{abstract}
\end{frontmatter}


\section{Introduction}
\label{sec:intro}

A central problem in additive combinatorics is to understand the inverse theory of the Gowers uniformity norms. Given a 1-bounded function $f\colon G\to\C$ with large Gowers uniformity norm where $G$ is a finite abelian group, the goal is to show that $f$ must correlate with some structured object.

This inverse problem has primarily received attention in the cases that $G$ is a cyclic group of prime order or a vector space over a finite field. For $G=\Z/N\Z$ the inverse problem was solved by Green, Tao, and Ziegler \cite{GTZ11, GTZ12} and the structured objects require the theory of nilsequences to describe. For vector spaces over finite fields the situation is known to be somewhat simpler. For $p\geq k$, a 1-bounded function $f\colon\F_p^n\to\C$ has large $U^{k+1}$-norm if and only if it correlates with a (classical) polynomial phase function, i.e., $e^{2\pi i P(x)/p}$ for a polynomial $P\colon\F_p^n\to\F_p$ of degree at most $k$, as shown by Bergelson, Tao, and Ziegler \cite{BTZ10, TZ10} (see also \cite{BSST21} for a discussion of the $p=k$ case). For $p<k$, one needs the theory of non-classical polynomials to describe the structured objects \cite{TZ12}.

The original proofs of these inverse theorems give extremely bad or even ineffective quantitative bounds. Recently Manners proved quantitative bounds for the $U^{k+1}$-inverse theorem over $\Z/N\Z$ \cite{Man18} and Gowers and Mili\'cevi\'c proved quantitative bounds for the $U^{k+1}$-inverse theorem over $\F_p^n$ when $p>k$ \cite{GM17, GM20}.

The goal of this paper is to build upon the work of Gowers and Mili\'cevi\'c to give quantitative bounds for the $U^4$-inverse theorem over $\F_p^n$ in the low-characteristic regime $p=2,3$. Our proof has two main steps: \emph{symmetrization} and \emph{integration} of low-characteristic trilinear forms. We are able to solve the integration problem for all $k$-linear forms, but the symmetrization problem we are only able to solve for trilinear forms.

\subsection{Statement of main result}

In this paper we always use $V$ to denote a finite-dimensional $\F_p$-vector space and we write $\omega=e^{2\pi i/p}$. Given a function $f\colon V\to\C$ and a shift $h\in V$, we write $\partial_h f(x)=f(x+h)\ol{f(x)}$ for the multiplicative derivative.

\begin{defn}
Given a function $f\colon V\to\C$ and $d\geq 2$, the {\textbf{Gowers uniformity norm}} $\|f\|_{U^d}$ is defined by \[\|f\|_{U^d}^{2^d}=\E_{x,h_1,\ldots,h_d\in V}(\partial_{h_1}\cdots\partial_{h_d}f)(x).\]
\end{defn}

This definition has many useful properties, including being a well-defined norm for all $d\geq 2$. Another useful property is the inductive formula $\|f\|_{U^d}^{2^d}=\E_h \|\partial_h f\|_{U^{d-1}}^{2^{d-1}}$.
See \cite[Lemma B.1]{TZ12} for more properties and references about the Gowers uniformity norms.

Consider a function $f\colon V\to G$ where $G$ is an abelian group. Given a shift $h\in V$, we write $\Delta_hf(x)=f(x+h)-f(x)$ for the additive derivative. It follows from the definition of the Gowers uniformity norm that a 1-bounded function $f\colon V\to\C$ satisfies $\|f\|_{U^d}\leq 1$ with equality if and only if $f=e^{2\pi i P}$ where $P\colon V\to\R/\Z$ satisfies $\Delta_{h_1}\Delta_{h_2}\cdots \Delta_{h_d}P(x)=0$. 

\begin{defn}
A \textbf{non-classical polynomial} of degree at most $k$ is a map $P\colon V\to\R/\Z$ that satisfies \[(\Delta_{h_1}\cdots \Delta_{h_{k+1}}P)(x)= 0\] for all $h_1,\ldots,h_{k+1},x\in V$.
We write $\Poly_{\leqslant k}(V\to\R/\Z)$ for the set of non-classical polynomials of degree at most $k$.
\end{defn}

A classical polynomial is a map $P\colon V\to\F_p$ satisfying the same definition. Thus we use $\Poly_{\leqslant k}(V\to\F_p)$ to denote the set of classical polynomials of degree at most $k$. In this paper we commonly abuse notation to identify $\F_p$ with $\{0,1/p,\ldots,(p-1)/p\}\subset\R/\Z$ when convenient. In this way we consider $\Poly_{\leqslant k}(V\to\F_p)$ to be a subset of $\Poly_{\leqslant k}(V\to\R/\Z)$.

\begin{thm}
\label{thm:inverse}
Fix a prime $p$ and $\delta>0$. There exists an $\epsilon>0$ satisfying \[1/\epsilon=\exp\paren{\exp\paren{\exp\paren{O_p\paren{\log(1/\delta)^{O(1)}}}}}\]
such that the following holds. Let $V$ be a finite-dimensional $\F_p$-vector space. Given a function $f\colon V\to\C$ satisfying $\|f\|_\infty\leq1$ and $\|f\|_{U^{4}}>\delta$, there exists a non-classical cubic polynomial $P\in\Poly_{\leqslant 3}(V\to\R/\Z)$ such that \[\abs{\E_{x\in V} f(x)e^{-2\pi i P(x)}}\geq \epsilon.\]
Furthermore, if $p\geq 3$, the polynomial can be taken to be classical.
\end{thm}

This theorem was proved with ineffective bounds by Tao and Ziegler \cite[Theorem 1.10]{TZ12} and with effective bounds for $p\geq 5$ by Gowers and Mili\'cevi\'c \cite[Theorem 7]{GM20}. The bounds we prove are of the same shape as Gowers and Mili\'cevi\'c's result.

\subsection{Proof strategy}
We start by summarizing the Gowers-Mili\'cevi\'c proof of the $U^4$-inverse theorem for $p\geq 5$ and then explain where the difficulties arise when $p<5$.

We are given a function $f\colon\F_p^n\to\C$ which is 1-bounded and satisfies $\|f\|_{U^4}>\delta$. By expanding the definition of the $U^4$-norm and applying the $U^2$-inverse theorem, this implies \[\delta^{16}<\|f\|_{U^4}^{16}=\E_{a,b}\|\partial_a\partial_b f\|_{U^2}^4\leq\E_{a,b}\|\widehat{\partial_a\partial_bf}\|_{\infty}^2.\] By an application of Markov's inequality, this implies that there is a fairly dense set $A\subset V^2$ and a function $\phi\colon A\to V$ such that $|\widehat{\partial_a\partial_bf}(\phi(a,b))|$ is large for each $(a,b)\in A$.

One can show that the function $\phi$ has some weak bilinearity properties. The main bulk of the work done by Gowers and Mili\'cevi\'c is to show that this local bilinearity can be upgraded to some global structure. Namely, they prove that $\phi$ agrees with a biaffine map $\Phi\colon V^2\to V$ on a fairly dense set. This immediately implies that there is a triaffine form $T\colon V^3\to\F_p$ that satisfies
\begin{equation}
\label{eq:GM-intro}
\abs{\E_{a,b,c,x}\partial_a\partial_b\partial_cf(x)\omega^{T(a,b,c)}}\geq\eta.
\end{equation}
This part of the proof has no assumptions on the characteristic $p$.

To finish the proof, we would like to find a cubic polynomial $P$ such that $\|f\omega^P\|_{U^3}$ is large. From this point we would be able to conclude by a single application of the $U^3$-inverse theorem. We expand
\[\|f\omega^P\|_{U^3}^8=\E_{a,b,c,x}\partial_a\partial_b\partial_cf(x)\omega^{\Delta_a\Delta_b\Delta_cP(x)}.\]
Thus the object $\Delta_a\Delta_b\Delta_cP(x)$ naturally appears. We call this object the total derivative of $P$. The total derivative of $P$ clearly does not depend on $x$ so we will just write $\Delta_a\Delta_b\Delta_cP$ from now on. Furthermore, $\Delta_a\Delta_b\Delta_cP$ is symmetric and trilinear in $a,b,c$.

Thus our next step is to massage \cref{eq:GM-intro} to turn the triaffine form $T$ into a symmetric trilinear form. For simplicity of exposition, let us assume that $T$ is a trilinear. (The parts of $T$ that are not trilinear can often be removed by an appropriate application of the Cauchy-Schwarz inequality.) We will first explain the argument in the $p\geq 5$ case.

\medskip

\noindent\textbf{Symmetrization:} We start with an argument based on an idea of Green and Tao. By several applications of the Cauchy-Schwarz inequality, \cref{eq:GM-intro} implies that $T(a,b,c)$ is ``close'' to $T(b,a,c)$, more precisely, we can prove that the map $(a,b,c)\mapsto T(a,b,c)-T(b,a,c)$ has bounded rank, and similarly for all other permutations of $a,b,c$.\footnote{See \cref{sec:preliminaries} for the definition of tensor rank that we use in this paper.}

To conclude the symmetrization step, we define
\begin{equation}
\label{eq:sym-intro}
S(a,b,c)=\frac16\paren{T(a,b,c)+T(a,c,b)+T(b,a,c)+T(b,c,a)+T(c,a,b)+T(c,b,a)}.
\end{equation}
Clearly $S$ is a symmetric trilinear form. Furthermore, by the Green-Tao argument, we know that $\rank(T-S)\ll_\eta 1$.

\medskip

\noindent\textbf{Integration:} We now integrate the symmetric trilinear form $S$ to a cubic polynomial $P$. Namely, define
\begin{equation}
\label{eq:int-intro}
P(x)=\frac16S(x,x,x).
\end{equation}
One can easily check that $\Delta_a\Delta_b\Delta_cP=S(a,b,c)$.

Starting from \cref{eq:GM-intro}, with some involved but not particularly interesting additional inequalities, we can conclude that $\|f\omega^P\|_{U^3}\gg_\eta1$. Applying the $U^3$-inverse theorem lets us complete the proof.

\subsection{Low characteristic}

Now we explain the difficulties that appear in characteristic $p<5$. Fortunately the bulk of the proof, up to \cref{eq:GM-intro}, works in arbitrary characteristic. The Green-Tao argument also goes through, showing that $T$ is close to each of its permutations. However, the remainder of the symmetrization step and the integration step, which were essentially trivial in high characteristic, fail badly when $p<5$. In particular, \cref{eq:sym-intro} and \cref{eq:int-intro} both involve dividing by 6.

For $p\geq 5$ we have the convenient characterization that $S$ satisfies $S(a,b,c)=\Delta_a\Delta_b\Delta_cP$ for some cubic polynomial $P$ if and only if $S$ is a symmetric trilinear form. However, this is no longer the case in low characteristic.

Following Tao and Ziegler, we define a classical symmetric trilinear form to be a trilinear form $S\colon V^3\to\F_p$ that satisfies $S(a,b,c)=\Delta_a\Delta_b\Delta_cP$ for some classical cubic polynomial $P$. As part of their proof of the inverse theorem Tao and Ziegler characterized classical symmetric $k$-linear forms for all $k$. In the same vein, we define a non-classical symmetric trilinear form to be $S\colon V^3\to\F_p$ that satisfies $S(a,b,c)=\Delta_a\Delta_b\Delta_cP$ for some non-classical cubic polynomial $P$. One problem we solve in this paper is the characterization of non-classical symmetric $k$-linear forms for all $k$.

As a special case, we will show that if $S$ is a symmetric trilinear form and also has some further symmetries which will be defined in \cref{sec:ncsm}, then we can integrate $S$ to a non-classical cubic polynomial $P$ (i.e., $\Delta_a\Delta_b\Delta_cP=S(a,b,c)$). We are able to characterize non-classical symmetric $k$-linear forms for all $k$, which solves the integration problem necessary for the $U^{k+1}$-inverse theorem for all $k$.

For the symmetrization problem, we have a trilinear form $T$ which we know is close in rank to all of its permutations. We wish to find, for $p=2$, an non-classical symmetric trilinear form $S$ such that $\rank(T-S)\ll 1$ and, for $p=3$, a classical symmetric trilinear $S$ such that $\rank(T-S)\ll 1$.

Even the problem of finding a symmetric $k$-linear $S$ that is close to $T$ is quite hard and we do not know how to solve it for general $k$. When $k=3$, the hypothesis that $T$ is close to all of its permutations implies that there is a subspace $U\leq V$ of bounded codimension so that $T|_U$ is symmetric. This lets us easily find a symmetric trilinear $S'$ such that $\rank(T-S')\ll 1$. For $p=3$, turning $S'$ into a classical symmetric trilinear form is easy. For $p=2$, we need to turn $S'$ into a non-classical symmetric trilinear form which takes some more work, but we are able to do so.

\medskip

\noindent\textbf{Structure of paper.} We give the necessary preliminaries in \cref{sec:preliminaries}. We solve the integration problem for all $k$ in \cref{sec:ncsm} and the symmetrization problem for $k=3$ in \cref{sec:symmetrization}. We combine these tools to prove our main result in \cref{sec:proof-of-main}. In \cref{sec:concl} we give some conjectures.

\section{Preliminaries}
\label{sec:preliminaries}

We gave a local definition of non-classical polynomials in \cref{sec:intro}. There is an equivalent global definition known. (Also see \cite[Lemma 1.7]{TZ12} for some more basic facts about non-classical polynomials.)

\begin{lemma}[{\cite[Lemma 1.7(iii)]{TZ12}}]
\label{poly-basis}
$P\colon \F_p^n\to\R/\Z$ is a non-classical polynomial of degree at most $k$ if and only if it can be expressed in the form
\[P(x_1,\ldots,x_n)=\alpha+\sum_{\genfrac{}{}{0pt}{}{0\leq i_1,\ldots,i_n<p,\, j\geq 0:}{0<i_1+\cdots+i_n\leq k-j(p-1)}}\frac{c_{i_1,\ldots,i_n,j}|x_1|^{i_1}\cdots|x_n|^{i_n}}{p^{j+1}}\pmod 1,\]
for some $\alpha\in\R/\Z$ and coefficients $c_{i_1,\ldots,i_n,j}\in\{0,\ldots,p-1\}$ where $|\cdot|$ is the standard map $\F_p\to\{0,\ldots,p-1\}$. Furthermore, this representation is unique.
\end{lemma}

For a tuple $(h_1,\ldots,h_k)\in V^k$ and a subset $I\subseteq[k]$, we use the notation $h_I$ to denote the tuple $(h_i)_{i\in I}\in V^I$.

\begin{defn}
Given a $k$-linear form $T\colon V^k\to\F_p$ we use two different notions for the rank of $T$. 

The \textbf{analytic rank} of $T$, denoted $\arank(T)$, is defined by the equation \[p^{-\arank(T)}=\E_{h_1,\ldots,h_k\in V}\omega^{T(h_1,\ldots,h_k)}.\]

The \textbf{partition rank} of $T$, denote $\prank(T)$, is the length of the shortest decomposition of $T$ as a sum of partition rank 1 forms. We define $T\colon V^k\to\F_p$ to have partition rank 1 if we can write $T(h_1,\ldots,h_k)=R(h_I)S(h_{[k]\setminus I})$ where $R\colon V^I\to\F_p$ and $S\colon V^{[k]\setminus I}\to\F_p$ are multilinear forms and $\emptyset\neq I\subsetneq[k]$.
\end{defn}

These definitions of rank enjoy many useful properties. One which will be important to us is the subadditivity of rank. Obviously $\prank(S+T)\leq\prank(S)+\prank(T)$. For analytic rank, we also have $\arank(S+T)\leq\arank(S)+\arank(T)$, but the proof is more subtle \cite[Theorem 1.5]{Lov19}.

For $k=2$, the partition rank and analytic rank coincide and both are equal to the usual notion of matrix rank. For general $k$ we have the following bounds.

\begin{thm}
\label{thm:analytic-partition}
There are constants $C_{p,k}$ and $D_k$ such that for a $k$-linear form $T\colon V^k\to\F_p$ we have the bounds
\[\arank(T)\leq\prank(T)\leq C_{p,k}\arank(T)^{D_k}.\]
\end{thm}

The lower bound is due to Lovett \cite{Lov19} and the upper bound was proved independently by Mili\'cevi\'c \cite{Mil19} and Janzer \cite{Jan20}. Note that a recent result of Cohen and Moshkovitz and independently Adiprasito, Kazhdan, and Ziegler improves the upper bound to linear when $k=3$ and $p>2$ \cite{CM21, AKZ21}.

One simple property of rank that we use multiple times through the argument is the following.

\begin{lemma}
\label{thm:close-in-rank}
Let $T\colon V^k\to\F_p$ be a $k$-linear form and let $U\leq V$ be a subspace such that $T|_U\equiv 0$. Then $\prank(T)\leq k\cdot\codim U$.
\end{lemma}

\begin{proof}
Write $n=\dim V$ and $r=\codim U$. Pick linearly independent linear forms $\alpha_1,\ldots,\alpha_n\colon V\to\F_p$ such that $\alpha_i|_U\equiv 0$ for $1\leq i\leq r$. (In other words $\alpha_1,\ldots,\alpha_n$ is a dual basis for $V$ and $\alpha_{r+1},\ldots,\alpha_n$ is a dual basis for $U$.) The tensor $T$ can be expressed in these coordinates as
\[T(x_1,\ldots,x_k)=\sum_{1\leq i_1,\ldots,i_k\leq n}T_{i_1,\ldots,i_k}\alpha_{i_1}(x_1)\cdots\alpha_{i_k}(x_k)\]for some coefficients $T_{i_1,\ldots,i_k}\in \F_p$ satisfying  $T_{i_1,\ldots,i_k}=0$ if $i_1,\ldots,i_k>r$.

Now one can easily group the above expression into the sum of $kr$ terms, each of partition rank 1. Explicitly,
\[T(x_1,\ldots,x_k)=\sum_{j=1}^k\sum_{\ell=1}^r \alpha_\ell(x_j)\beta_{j,\ell}(x_{[k]\setminus j})\]where
\[\beta_{j,\ell}(x_{[k]\setminus j})=\sum_{\genfrac{}{}{0pt}{}{r<i_1,\ldots,i_{j-1}\leq n}{1\leq i_{j+1},\ldots, i_n\leq n}}T_{i_1,\ldots,i_{j-1},\ell,i_{j+1},\ldots,i_k}\alpha_{i_1}(x_1),\ldots,\alpha_{i_{j-1}}(x_{j-1})\alpha_{i_{j+1}}(x_{j+1})\cdots\alpha_{i_k}(x_k).\qedhere\]
\end{proof}

\section{Non-classical symmetric multilinear forms}
\label{sec:ncsm}

In this section we define non-classical symmetric multilinear forms ($\nCSM$ for short) and classical symmetric multilinear forms ($\CSM$ for short) as multilinear forms with additional symmetry properties. Then we will show that these objects are exactly the total derivatives of non-classical polynomials and classical polynomials respectively.

\begin{defn}
A \textbf{non-classical symmetric multilinear form} is a map $T\colon V^k\to\F_p$ that is:
\begin{itemize}
    \item (multilinear) fixing all the variables but one, the map $h_i\mapsto T(h_1,\ldots,h_k)$ is a linear map;
    \item (symmetric) $T(h_1,\ldots,h_k)$ is invariant under permutations of the $k$ variables $h_1,\ldots, h_k$;
    \item (non-classical) $T(h_1,\ldots, h_1,h_2,\ldots, h_{k-p+1})=T(h_1,h_2,\ldots,h_2,h_3,\ldots, h_{k-p+1})$ where the variable $h_1$ appears $p$ times in the first expression and $h_2$ appears $p$ times in the second expression.
\end{itemize}
We write $\nCSM^k(V)$ for the space of non-classical symmetric multilinear forms $T\colon V^k\to \F_p$.

A \textbf{classical symmetric multilinear form} additionally satisfies:
\begin{itemize}
    \item (classical) $T(h_1,\ldots, h_1,h_2,\ldots, h_{k-p+1})=0$ where the variable $h_1$ appears $p$ times in the left-hand side.
\end{itemize}
We write $\CSM^k(V)$ for the space of classical symmetric multilinear forms $T\colon V^k\to \F_p$.
\end{defn}

Note that the classical condition is stronger than the non-classical condition, so $\CSM^k(V)\subseteq\nCSM^k(V)$. For $k\leq p$ the non-classical condition is vacuous while for $k<p$ the classical condition is vacuous. In their proof of the inverse theorem Tao and Ziegler introduced classical symmetric multilinear forms and (implicitly) non-classical symmetric multilinear forms \cite{TZ12}.

\begin{defn}
For $P\in\Poly_{\leqslant k}(V\to\R/\Z)$, define the \textbf{total derivative} of $P$ to be the map $d^kP\colon V^k\to\F_p$ defined by \[d^kP(h_1,\ldots, h_k)=(\Delta_{h_1}\cdots \Delta_{h_k}P)(0).\] The right-hand side of this equation lies in $\{0,1/p,\ldots,(p-1)/p\}\subset\R/\Z$ which we identify with $\F_p$; furthermore $d^kP(h_1,\ldots, h_k)=(\Delta_{h_1}\cdots \Delta_{h_k}P)(x)$ for all $h_1,\ldots, h_k,x\in V$.\footnote{One way to see these facts is to notice that $\Delta_{h_2}\cdots\Delta_{h_k}P\in\Poly_{\leqslant 1}(V\to\R/\Z)$. By \cref{poly-basis} this expression is of the form $\alpha+c_1|x_1|/p+\cdots+c_n|x_n|/p \pmod1$. Applying $\Delta_{h_1}$, we see that $\Delta_{h_1}\cdots \Delta_{h_k}P(x)=(h_{1,1}c_1+\cdots+h_{1,n}c_n)/p$ is independent of $x$ and lies in the set $\{0,1/p,\ldots,(p-1)/p\}$.}
\end{defn}

Tao and Ziegler prove that the total derivatives of classical polynomials are exactly CSMs.

\begin{prop}[{\cite[Lemma 4.5]{TZ12}}]
\label{thm:integrate-csm}
For $k\geq 1$ and a classical polynomial $P\in\Poly_{\leqslant k}(V\to\F_p)$, the total derivative $d^kP$ lies in $\CSM^k(V)$. Furthermore this map is surjective. In other words, we have the short exact sequence
\[0\to \Poly_{\leqslant k-1}(V\to\F_p)\to\Poly_{\leqslant k}(V\to\F_p)\xrightarrow{d^k}\CSM^k(V)\to0.\]
\end{prop}

They also show that the total derivative of a non-classical polynomial is an nCSM.

\begin{prop}[{\cite[Eq. (4.1)]{TZ12}}]
For $P\in\Poly_{\leqslant k}(V\to\R/\Z)$, the total derivative $d^kP$ lies in $\nCSM^k(V)$.
\end{prop}

Our integration result gives the implication in the opposite implication.

\begin{prop}[Integrating $\nCSM$s]
\label{thm:integrate-ncsm}
For $k\geq 1$ and $T\in\nCSM^k(V)$, there exists a non-classical polynomial $P\in\Poly_{\leqslant k}(V\to\R/\Z)$ such that $d^kP=T$. In other words, we have the short exact sequence
\[0\to \Poly_{\leqslant k-1}(V\to\R/\Z)\to\Poly_{\leqslant k}(V\to\R/\Z)\xrightarrow{d^k}\nCSM^k(V)\to0.\]
\end{prop}

\begin{proof}
First note that $\ker(d^k)=\Poly_{\leqslant k-1}(V\to\R/\Z)$ since $\Delta_{h_1}\cdots \Delta_{h_k}P$ is identically zero for all $h_1,\ldots, h_k$ if and only if $P$ is a non-classical polynomial of degree at most $k-1$.

We show that $d^k$ is surjective in an inexplicit manner. First define $\ol{\Poly}_{\leqslant k}(V\to\R/\Z)$ to be the quotient of $\Poly_{\leqslant k}(V\to\R/\Z)$ under the equivalence relation $P\sim P+\alpha$ for $\alpha\in\R/\Z$. Note that for $k\geq 1$, the total derivative $d^k$ acts on this quotient space. We will prove that \begin{equation}
\label{integrate-counting}
\frac{|\ol{\Poly}_{\leqslant k}(V\to\R/\Z)|}{|\ol{\Poly}_{\leqslant k-1}(V\to\R/\Z)|}=|\nCSM^k(V)|.
\end{equation}
This suffices to complete the proof.

Choose an isomorphism $V\cong\F_p^n$. By \cref{poly-basis}, we see that $|\ol{\Poly}_{\leqslant k}(V\to\R/\Z)|=p^{C'_k}$ where $C'_k$ is defined to be the number of tuples $(i_1,\ldots,i_n,j)$ where $i_1,\ldots, i_n\in\{0,\ldots, p-1\}$ and $j\geq 0$ satisfy $0<i_1+\cdots+i_n\leq k-j(p-1)$. Then the left-hand side of \cref{integrate-counting} is $p^{C'_k-C'_{k-1}}$. Define $C_k=C'_k-C'_{k-1}$. We can easily see that $C_k$ is the number of tuples $(i_1,\ldots,i_n,j)$ where $i_1,\ldots, i_n\in\{0,\ldots, p-1\}$ and $0\leq j\leq (k-1)/(p-1)$ satisfy $i_1+\cdots+i_n= k-j(p-1)$. 

Now we compute $|\nCSM^k(V)|$. Clearly any multilinear form $T\colon V^k\to\F_p$ can be expressed uniquely as \[T(x_1,\ldots, x_k)=\sum_{j_1,\ldots, j_k\in[n]}c_{j_1,\ldots, j_k}x_{1,j_1}\cdots x_{k,j_k}\]where $c_{j_1,\ldots,j_k}\in\F_p$. For a tuple $\bm j=(j_1,\ldots, j_k)$, define $i(\bm j)=(i_1,\ldots, i_n)$ such that $i_\ell$ is the number of $j_1,\ldots, j_k$ which are equal to $\ell$. With this notation, we can say that a multilinear form $T$ is symmetric if and only if $c_{\bm j_1}=c_{\bm j_2}$ whenever $i(\bm j_1)=i(\bm j_2)$.

Finally define $i'(\bm j)=(i'_1,\ldots, i'_n)$ where $i'_\ell=0$ if $i_\ell=0$ and otherwise $i'_\ell\in\{1,\ldots, p-1\}$ satisfies $i'_{\ell}\equiv i_{\ell}\pmod{p-1}$. We can see that a multilinear form $T$ is an $\nCSM$ if and only if $c_{\bm j_1}=c_{\bm j_2}$ whenever $i'(\bm j_1)=i'(\bm j_2)$. Thus $|\nCSM^k(V)|=p^{D_k}$ where $D_k$ is the number of tuples $\bm i'\in\{0,\ldots,p-1\}^n$ such that there exists $\bm j\in[n]^k$ such that $\bm i'=i'(\bm j)$. Clearly the only constraint on such a tuple $\bm i'\in\{0,\ldots,p-1\}^n$ is that $s=i'_1+\cdots +i'_n$ satisfies $0<s\leq n$ and $s\equiv n\pmod{p-1}$. Thus we see that $C_k=D_k$, proving the desired result.
\end{proof}

\section{Symmetrization in low characteristic}
\label{sec:symmetrization}

For a $k$-linear form $T\colon V^k\to \F_p$ and a permutation $\pi\in\mathfrak S_k$, write $T_\pi\colon V^k\to\F_p$ for \[T_\pi(x_1,\ldots, x_k)=T(x_{\pi(1)},\ldots,x_{\pi(k)}).\]

We have one tool that lets us show that forms are close to symmetric. The following result is based on an idea of Green and Tao.

\begin{lemma}[{cf. \cite[Lemma 2.8]{Gre07}}]
\label{thm:gt-sym}
Let $A\colon V^2\to\F_p$ be a bilinear form. Let $b_1,b_2,b_3\colon V\to\C$ be arbitrary 1-bounded functions. If\[\abs{\E_{u,v\in V}b_1(u)b_2(v)b_3(u+v)\omega^{A(u,v)}}\geq\delta,\]then\[p^{-\arank(A-A_{(12)})}=\E_{u,v}\omega^{A(u,v)-A(v,u)}\geq\delta^8.\]
\end{lemma}

\begin{proof}
By an application of the Cauchy-Schwarz inequality and the 1-boundedness of $b_2$, we have
\[\delta^2\leq\E_{u,u',v\in V}b_1(u)\ol{b_1(u')}b_3(u+v)\ol{b_3(u'+v)}\omega^{A(u-u',v)}.\]

We make the change of variables $u''=u+u'+v$. Rearranging, we obtain
\begin{align*}
\delta^2
&\leq\E_{u,u',u''}b_1(u)\ol{b_1(u')}b_3(u''-u')\ol{b_3(u''-u)}\omega^{A(u-u',u''-u-u')}\\
&=\E_{u,u',u''}\paren{b_1(u)\ol{b_3(u''-u)}\omega^{A(u,u''-u)}}\paren{\ol{b_1(u')}b_3(u''-u')\omega^{A(-u',u''-u')}}\omega^{A(u',u)-A(u,u')}.
\end{align*}

Write $B=A-A_{(12)}$. By first averaging over $u''$ and then choosing $b_1',b_2'\colon V\to\C$ appropriate 1-bounded functions, the above inequality becomes
\[\delta^2\leq\abs{\E_{u,u'}b_1'(u)b_2'(u')\omega^{B(u,u')}}.\]

By two more applications of the Cauchy-Schwarz inequality, we conclude
\[\delta^8\leq\E_{u,v,u',v'}\omega^{B(u-v,u'-v')}=\E_{u,u'}\omega^{B(u,u')}=p^{-\arank(B)}.\qedhere\]
\end{proof}

\begin{cor}
\label{thm:gt-symmetry-cor}
Let $T\colon V^3\to\F_p$ be a trilinear form and let $b_1,\ldots,b_7\colon V\to\C$ be 1-bounded functions such that
\[\abs{\E_{x,y,z\in V}b_1(x)b_2(y)b_3(z)b_4(x+y)b_5(x+z)b_6(y+z)b_7(x+y+z)(-1)^{T(x,y,z)}}\geq\delta.\]
Then we have $\arank(T-T_\pi)\leq16\log_p(1/\delta)$ for all $\pi\in\mathfrak S_3$.
\end{cor}

\begin{proof}
Suppose $\pi=(12)$. By the triangle inequality, we see that there exists a function $\delta\colon V\to[0,1]$ and 1-bounded functions $b_{1,z},b_{2,z},b_{3,z}\colon V\to\C$ such that
\[\abs{\E_{x,y}b_{1,z}(x)b_{2,z}(y)b_{3,z}(x+y)\omega^{T(x,y,z)}}\geq\delta(z)\] for all $z\in V$ and additionally $\E_{z\in V}\delta(z)\geq \delta$.

By \cref{thm:gt-sym} for all $z\in V$, \[\E_{x,y}\omega^{T(x,y,z)-T(y,x,z)}\geq\delta(z)^8.\] Averaging over $z$ and applying convexity,
\[p^{-\arank(T-T_{(12)})}=\E_{x,y,z}\omega^{T(x,y,z)-T(y,x,z)}\geq\E_{z}\delta(z)^8\geq\delta^8,\]so $\arank(T-T{(12)})\leq 8\log_p(1/\delta)$. The same holds for all transpositions. Since every permutation in $\mathfrak S_3$ can be written as the product of at most two transpositions, we conclude by the subadditivity of analytic rank. 
\end{proof}

\begin{prop}
\label{thm:first-sym}
Let $T\colon V^3\to \F_p$ be a trilinear form. Suppose that $T$ is close to symmetric in the sense that $\prank(T-T_\pi)\leq r$ for all permutations $\pi\in\mathfrak S_3$. Then there exists a subspace $U\leq V$ with $\codim U\leq 5r$ such that $T|_U\colon U^3\to\F_p$ is a symmetric trilinear form.
\end{prop}

\begin{proof}
By definition, there exists a partition rank decomposition of each $T-T_\pi$ (for $\pi\in\mathfrak S_3\setminus\{\id\}$) as the sum of at most $r$ terms, each of which is the product of a linear form and a bilinear form. Let $U$ be the subspace of $V$ where all of the linear forms involved in this decomposition vanish. Since there are at most $5r$ linear forms in this decomposition, we have the desired bound on the codimension of $U$. Furthermore, since $(T-T_\pi)|_U$ is identically zero for each $\pi\in\mathfrak S_3$, we see that $T|_U$ is symmetric.
\end{proof}

\begin{rem}
This result immediately implies that there exists a symmetric trilinear form $S\colon V^3\to\F_p$ with $\prank(T-S)\leq 15r$. To see this, pick an arbitrary decomposition $V=U\oplus W$. Then define $S$ by $S(u_1+ w_1,u_2+ w_2,u_3+ w_3)=T|_U(u_1,u_2,u_3)$ for all $u_1,u_2,u_3\in U$ and $w_1,w_2,w_3\in W$. Since $T|_U$ is a symmetric trilinear form it follows that $S$ is a symmetric trilinear form. By \cref{thm:close-in-rank},  we have the bound $\prank(T-S)\leq 15r$.
\end{rem}

\cref{thm:first-sym} easily handles the initial symmetrization step. To complete the symmetrization, we need to find some $R\in\nCSM^3(V)$ such that $\prank(T-R)$ is small. For $p\geq 5$ this is trivial, since every symmetric trilinear form is an $\nCSM$ and actually a $\CSM$.

For $p=3$ every symmetric trilinear form is also an $\nCSM$ though not necessarily a $\CSM$. We note however, that for $p=3$ every $\nCSM$ is close to a $\CSM$. This allows us to prove the $U^4$-inverse theorem over $\F_3^n$ with classical polynomials.

\begin{prop}
\label{thm:f3-csm}
Let $T\colon V^3\to\F_3$ be a symmetric trilinear form. Then there exists a codimension 1 subspace $U\leq V$ such that $T|_U\in\CSM^3(U)$.
\end{prop}

\begin{proof}
The only additional condition that $T$ needs to satisfy to be a $\CSM$ is $T(x,x,x)=0$. Thus if $U$ is any subspace of $V$ such that $T(x,x,x)=0$ for all $x\in U$, we would have $T|_U\in\CSM^3(U)$.

Consider the map $x\mapsto T(x,x,x)\in\F_3$. Because $T$ is a symmetric trilinear form, this map is linear. Indeed, $T(x+y,x+y,x+y)=T(x,x,x)+3T(x,x,y)+3T(x,y,y)+T(y,y,y)=T(x,x,x)+T(y,y,y)$. Defining $U$ to be the codimension 1 subspace on which $T(x,x,x)$ vanishes, we have the desired conclusion.
\end{proof}

Applying \cref{thm:first-sym}, \cref{thm:f3-csm}, and \cref{thm:close-in-rank} we conclude the following.

\begin{cor}
\label{thm:f3-csm-cor}
Fix $p\geq 3$. Let $T\colon V^3\to\F_p$ be a trilinear form that satisfies $\prank(T-T_\pi)\leq r$ for all permutations $\pi\in\mathfrak S_3$. Then there exists $S\in\CSM^3(V)$ such that $\prank(T-S)\leq 15r+3$.
\end{cor}

For $p=2$ the situation is more complicated. It is known, due to examples of Green and Tao \cite{GT09} and Lovett, Meshulam, and Samorodnitsky \cite{LMS11} that classical polynomials are not sufficient for the $U^4$-inverse theorem over $\F_2$. This means that we will not be able to find a $\CSM$ that is close to $T$. However the following argument shows that $T$ is close to an $\nCSM$.

\begin{prop}
\label{thm:f2-ncsm}
Let $T\colon V^3\to\F_2$ be a symmetric trilinear form. Suppose there are 1-bounded functions $b_1,\ldots,b_7\colon V\to\C$ such that
\[\abs{\E_{x,y,z\in V}b_1(x)b_2(y)b_3(z)b_4(x+y)b_5(x+z)b_6(y+z)b_7(x+y+z)(-1)^{T(x,y,z)}}\geq\delta.\]
Then there exists a subspace $U\leq V$ satisfying $\codim U\leq 8\log_2(1/\delta)$ such that $T|_U\in\nCSM^3(U)$.
\end{prop}

\begin{proof}
Make the change of variables $z=x+w$. Then the hypothesis becomes
\begin{align*}
\delta
&\leq\abs{\E_{x,y,w\in V}b_1(x)b_2(y)b_3(x+w)b_4(x+y)b_5(w)b_6(x+y+w)b_7(y+w)(-1)^{T(x,x,y)+T(x,y,w)}}\\
&=\abs{\E_w b_5(w)\E_{x,y}(b_1(x)b_3(x+w))(b_2(y)b_7(y+w))(b_4(x+y)b_6(x+y+w))(-1)^{T(x,x,y)+T(x,y,w)}}.
\end{align*}

By averaging (and the 1-boundedness of $b_5$), there exists $w_0\in V$ such that \[\delta\leq\abs{\E_{x,y}(b_1(x)b_3(x+w_0))(b_2(y)b_7(y+w_0))(b_4(x+y)b_6(x+y+w_0))(-1)^{T(x,x,y)+T(x,y,w_0)}}.\]

Consider the map $A\colon V^2\to\F_2$ defined by $A(x,y)=T(x,x,y)+T(x,y,w_0)$. We claim that $A$ is bilinear. Using the symmetry and trilinearity of $T$,
\begin{align*}
A(x+x',y)
&=T(x+x',x+x',y)+T(x+x',y,w_0)\\
&=T(x,x,y)+2T(x,x',y)+T(x',x',y)+T(x,y,w_0)+T(x',y,w_0)\\
&=A(x,y)+A(x',y).
\end{align*}
Similarly,
\begin{align*}
A(x,y+y')
&=T(x,x,y+y')+T(x,y+y',w_0)\\
&=T(x,x,y)+T(x,x,y')+T(x,y,w_0)+T(x,y',w_0)\\
&=A(x,y)+A(x,y').
\end{align*}

Since $A$ is bilinear, we can apply \cref{thm:gt-sym} to conclude that 
\[\delta^8\leq\E_{x,y}(-1)^{A(x,y)-A(y,x)}.\] Note that 
\[A(x,y)-A(y,x)=T(x,x,y)+T(x,y,w_0)-T(y,y,x)-T(y,x,w_0)=T(x,x,y)-T(x,y,y)\] by the symmetry of $T$.

Define $B\colon V^2\to\F_2$ by $B(x,y)=T(x,x,y)-T(x,y,y)$. Note that $B$ is bilinear. (This can be seen since $B(x,y)=A(x,y)-A(y,x)$ or by direct computation.) We have shown that $\arank B\leq 8\log_2(1/\delta)$. Let $U$ be the nullspace of $B$.\footnote{For the purpose of the following argument we could take $U$ to either be the left or right nullspace of $B$, though since $B$ is antisymmetric its left and right nullspaces agree.} This satisfies $\codim U=\prank B=\arank B\leq 8\log_2 (1/\delta)$. Furthermore $B(x,y)=0$ for all $x,y\in U$, which implies that $T|_U$ is an $\nCSM$.
\end{proof}

This essentially implies the desired symmetrization result. To put the result in the form we need, we use the following application of the Cauchy-Schwarz inequality which will also appear again later in the paper.

\begin{lemma}
\label{thm:multiaffine-cs}
Let $\phi\colon V^3\to\F_p$ be a triaffine form. Suppose there are 1-bounded functions $b_1,\ldots,b_7\colon V\to\C$ such that
\[\abs{\E_{x,y,z\in V}b_1(x)b_2(y)b_3(z)b_4(x+y)b_5(x+z)b_6(y+z)b_7(x+y+z)\omega^{\phi(x,y,z)}}\geq\delta.\]
Let $T\colon V^3\to\F_p$ be the trilinear part of $\phi$. Then there exist 1-bounded functions $b'_1,\ldots,b'_7\colon V\to\C$ such that
\[\abs{\E_{x,y,z\in V}b'_1(x)b'_2(y)b'_3(z)b'_4(x+y)b'_5(x+z)b'_6(y+z)b'_7(x+y+z)\omega^{T(x,y,z)}}\geq\delta^8.\]
\end{lemma}

\begin{proof}
This result follows from three applications of the Cauchy-Schwarz inequality, one in each variable. We have 
\begin{align*}
\delta^2
&\leq\abs{\E_{y,z}b_2(y)b_3(z)b_6(y+z)\E_x b_1(x)b_4(x+y)b_5(x+z)b_7(x+y+z)\omega^{\phi(x,y,z)}}^2\\
&\leq \paren{\E_{y,z}\abs{b_2(y)b_3(z)b_6(y+z)}^2}\\
&\qquad\qquad\cdot\left(\E_{x,x',y,z}b_1(x)\ol{b_1(x')}b_4(x+y)\ol{b_4(x'+y)}b_5(x+z)\ol{b_5(x'+z)}\right.\\
&\qquad\qquad\qquad\cdot\left.b_7(x+y+z)\ol{b_7(x'+y+z)}\omega^{\phi(x,y,z)-\phi(x',y,z)}\right).
\end{align*}
The 1-boundedness of $b_2,b_3,b_6$ lets us remove the first term. 

The next application of the Cauchy-Schwarz inequality gives
\begin{align*}
\delta^4
&\leq \E_{x,x',y,y',z}b_4(x+y)\ol{b_4(x'+y)}\ol{b_4(x+y')}{b_4(x'+y')}b_7(x+y+z)\ol{b_7(x'+y+z)}\\
&\qquad\qquad\cdot\ol{b_7(x+y'+z)}{b_7(x'+y'+z)}\omega^{\phi(x,y,z)-\phi(x',y,z)-\phi(x,y',z)+\phi(x',y',z)}.
\end{align*}
The third gives
\begin{align*}
\delta^8
&\leq \E_{x,x',y,y',z,z'}b_7(x+y+z)\ol{b_7(x'+y+z)}\ol{b_7(x+y'+z)}{b_7(x'+y'+z)}\\
&\qquad\qquad\cdot\ol{b_7(x+y+z')}{b_7(x'+y+z')}{b_7(x+y'+z')}\ol{b_7(x'+y'+z')}\\
&\qquad\qquad\cdot\omega^{\phi(x,y,z)-\phi(x',y,z)-\phi(x,y',z)+\phi(x',y',z)-\phi(x,y,z')+\phi(x',y,z')+\phi(x,y',z')-\phi(x',y',z')}.
\end{align*}

Since $T$ is defined to be the trilinear part of the triaffine form $\phi$, the exponent in the last term is equal to $T(x-x',y-y',z-z')$. Applying a change of variables and averaging, we can find $x_0,y_0,z_0$ such that
\begin{align*}
\delta^8
&\leq \E_{x,y,z}b_7(x+y+z+x_0+y_0+z_0)\ol{b_7(y+z+x_0+y_0+z_0)}\ol{b_7(x+z+x_0+y_0+z_0)}{b_7(z+x_0+y_0+z_0)}\\
&\qquad\qquad\cdot\ol{b_7(x+y+x_0+y_0+z_0)}{b_7(y+x_0+y_0+z_0)}{b_7(x+x_0+y_0+z_0)}\ol{b_7(x_0+y_0+z_0)}\omega^{T(x,y,z)}.
\end{align*}
This expression is of the desired form, completing the proof.
\end{proof}

\begin{cor}
\label{thm:f2-ncsm-cor}
Let $T\colon V^3\to\F_2$ be a trilinear form and let $b_1,\ldots,b_7\colon V\to\C$ be 1-bounded functions such that
\[\abs{\E_{x,y,z\in V}b_1(x)b_2(y)b_3(z)b_4(x+y)b_5(x+z)b_6(y+z)b_7(x+y+z)(-1)^{T(x,y,z)}}\geq\delta.\]
Then there exists $S\in\nCSM^3(V)$ such that $\prank(T-S)\leq 432\log_2(1/\delta)$.
\end{cor}

\begin{proof}
By \cref{thm:gt-symmetry-cor}, we have that $\arank(T-T_\pi)\leq 16\log_2(1/\delta)$ for all $\pi\in\mathfrak S_3$. Then applying \cref{thm:first-sym}, there exists a subspace $U\leq V$ with $\codim U\leq 80\log_2(1/\delta)$ such that $T|_U\colon U^3\to\F_2$ is a symmetric trilinear form.

By averaging, there exist cosets $x_0+U,y_0+U,z_0+U$ such that
\begin{align*}
\delta
&\leq\abs{\E_{\substack{x\in x_0+U\\y\in y_0+U\\z\in z_0+U}}b_1(x)b_2(y)b_3(z)b_4(x+y)b_5(x+z)b_6(y+z)b_7(x+y+z)(-1)^{T(x,y,z)}}\\
&=\abs{\E_{x,y,z\in U}b'_1(x)b'_2(y)b'_3(z)b'_4(x+y)b'_5(x+z)b'_6(y+z)b'_7(x+y+z)(-1)^{T(x_0+x,y_0+y,z_0+z)}}.
\end{align*}
Now by \cref{thm:multiaffine-cs}, this implies that 
\[\delta^8\leq\abs{\E_{x,y,z\in U}b''_1(x)b''_2(y)b''_3(z)b''_4(x+y)b''_5(x+z)b''_6(y+z)b''_7(x+y+z)(-1)^{T(x,y,z)}}.\]

We can now apply \cref{thm:f2-ncsm} to find a subspace $W\leq U$ such that $\codim_V W\leq 144\log_2(1/\delta)$ and such that $T|_W\in\nCSM^3(W)$. To conclude, all that remains is to extend $T|_W$ to an $\nCSM$ on all of $V$. To do this we first pick an arbitrary decomposition $V=W\oplus W'$. For $w_1,w_2,w_3\in W$ and $w_1',w_2',w_3'\in W'$ define $S(w_1+ w'_1,w_2+ w'_2,w_3+ w'_3)=T|_W(w_1,w_2,w_3)$. Clearly $S\in\nCSM^3(V)$ and $\prank(T-S)\leq 3\codim_V W\leq 432\log_2(1/\delta)$ by \cref{thm:close-in-rank}.
\end{proof}

\section{Proof of main theorem}
\label{sec:proof-of-main}

We now have the tools to prove the main theorem. We first record the result of Gowers and Mili\'cevi\'c that we require. This result can be found at the beginning of Section 10 of \cite{GM17} and its proof takes up the bulk of that paper.

\begin{thm}[{\cite{GM17}}]
\label{thm:gm-multiaffine-corr}
Fix a prime $p$ and $\delta>0$. There exists \[\epsilon=\paren{\exp\paren{\exp\paren{\exp\paren{O_p\paren{\log(1/\delta)^{O(1)}}}}}}^{-1}\]
such that the following holds. Let $V$ be a finite-dimensional $\F_p$-vector space. Given a function $f\colon V\to\C$ satisfying $\|f\|_\infty\leq1$ and $\|f\|_{U^{4}}>\delta$, there exists a triaffine form $\phi\colon V^3\to\F_p$ such that
\begin{equation}
\label{eq:tri-affine-cor}
\abs{\E_{x,h_1,h_2,h_3\in V}(\partial_{h_1}\partial_{h_2}\partial_{h_3}f)(x)\omega^{\phi(h_1,h_2,h_3)}}\geq\epsilon.
\end{equation}
\end{thm}

\begin{proof}[Proof of \cref{thm:inverse}]
Applying \cref{thm:gm-multiaffine-corr}, we find a triaffine form $\phi\colon V^3\to\F_p$ such that
\[\abs{\E_{x,h_1,h_2,h_3\in V}(\partial_{h_1}\partial_{h_2}\partial_{h_3}f)(x)\omega^{\phi(h_1,h_2,h_3)}}\geq\epsilon.\]

Let $T$ be the trilinear part of $\phi$. First average over $x$ and expand out the derivative to conclude that there exist 1-bounded functions $b_1,\ldots,b_7\colon V\to\C$ such that
\[\abs{\E_{x,y,z\in V}b_1(x)b_2(y)b_3(z)b_4(x+y)b_5(x+z)b_6(y+z)b_7(x+y+z)\omega^{\phi(x,y,z)}}\geq\epsilon.\]
Then applying \cref{thm:multiaffine-cs} and then \cref{thm:gt-symmetry-cor}, we see that $\arank(T-T_\pi)\leq 128\log_p(1/\epsilon)$ for all $\pi\in\mathfrak S_3$. By \cref{thm:analytic-partition}, we conclude that $\prank(T-T_\pi)\leq O_p(\log_p(1/\epsilon)^{O(1)})$ for all $\pi\in\mathfrak S_3$.

Now for $p\geq 3$, we can apply \cref{thm:f3-csm-cor} to find $S\in\CSM^3(V)$ such that $\prank(T-S)\leq O_p(\log_p(1/\epsilon)^{O(1)})$. Then by \cref{thm:integrate-csm}, there exists a classical cubic polynomial $P\in\Poly_{\leqslant 3}(V\to\F_p)$ such that $d^3P=S$. For $p=2$, we apply \cref{thm:f2-ncsm-cor} to find $S\in\nCSM^3(V)$ such that $\prank(T-S)\leq O_p(\log_p(1/\epsilon)^{O(1)})$. Then by \cref{thm:integrate-ncsm}, there exists a non-classical cubic polynomial $P\in\Poly_{\leqslant 3}(V\to\R/\Z)$ such that $d^3P=S$.

In either case we define $r=\max\{\prank(T-S),\log_p(1/\epsilon)\}$ satisfying $r\leq O_p(\log_p(1/\epsilon)^{O(1)})$. We have a decomposition $T-S=\sum_{i=1}^r\gamma_i$ where each $\gamma_i\colon V^3\to\F_p$ is the product of a linear form in one of the variables with a bilinear form in the other two variables. Write $\Gamma\colon V^3\to\F_p^{2r}$ for the list of $2r$ linear and bilinear forms that factor the $\gamma_i$'s. For example, if $\gamma_1(x,y,z)=\alpha(x)\beta(y,z)$, then we define $\Gamma_1(x,y,z)=\alpha(x)$ and $\Gamma_2(x,y,z)=\beta(y,z)$.

Define $g=f\omega^P$. Note that $(\partial_{h_1}\partial_{h_2}\partial_{h_3}g)(x)=(\partial_{h_1}\partial_{h_2}\partial_{h_3}f)(x)\omega^{S(h_1,h_2,h_3)}$. The inequality we started with can now be written as
\[\abs{\E_{x,h_1,h_2,h_3}b(h_1,h_2,h_3)\omega^{\sum_{i=1}^r\gamma_i(h_1,h_2,h_3)}(\partial_{h_1}\partial_{h_2}\partial_{h_3}g)(x)}\geq \epsilon\]
where $b(h_1,h_2,h_3)$ is a product of linear and bilinear phase functions. We first remove the middle term.

By averaging, there exists $c\in\F_p^{2r}$ such that
\begin{equation}
\label{eq:gamma-correlation}
\abs{\E_{x,h_1,h_2,h_3}b(h_1,h_2,h_3)\mathbbm1(\Gamma(h_1,h_2,h_3)=c)(\partial_{h_1}\partial_{h_2}\partial_{h_3}g)(x)}\geq \epsilon p^{-2r}.
\end{equation}
Now by Fourier analysis, we can write
\[\mathbbm1(\Gamma(h_1,h_2,h_3)=c)=\E_{\xi\in\F_p^{2r}}\omega^{(\Gamma(h_1,h_2,h_3)-c)\cdot \xi}.\] Plugging this expression into \cref{eq:gamma-correlation}, by averaging again, there exists $\xi_0\in\F_p^{2r}$ such that 
\[\abs{\E_{x,h_1,h_2,h_3}b(h_1,h_2,h_3)\omega^{(\Gamma(h_1,h_2,h_3)-c)\cdot\xi_0}(\partial_{h_1}\partial_{h_2}\partial_{h_3}g)(x)}\geq \epsilon p^{-2r}.\] Since $\Gamma$ is a list of linear and bilinear forms, the middle term, $\omega^{(\Gamma(h_1,h_2,h_3)-c)\cdot\xi_0}$, is a product of linear and bilinear phase functions. Thus we conclude that \[\abs{\E_{x,h_1,h_2,h_3}b'(h_1,h_2,h_3)(\partial_{h_1}\partial_{h_2}\partial_{h_3}g)(x)}\geq \epsilon p^{-2r},\]where $b'$ is a product of linear and bilinear phase functions.

Write \[b'(h_1,h_2,h_3)=\omega^{\beta_1(h_2,h_3)+\beta_2(h_1,h_3)+\beta_3(h_1,h_2)+\alpha_1(h_1)+\alpha_2(h_2)+\alpha_3(h_3)}.\] By averaging, there exist $x,h_1\in V$ such that
\[\abs{\E_{h_2,h_3}b'(h_1,h_2,h_3)(\partial_{h_1}\partial_{h_2}\partial_{h_3}g)(x)}\geq \epsilon p^{-2r}.\] Expanding this out, we see that this inequality is of the correct form to apply \cref{thm:gt-sym} which then implies that $\arank(\beta_1-(\beta_1)_{(12)})\leq 8(2r+\log_p(1/\epsilon))\leq 24r$. The same is true of $\beta_2,\beta_3$.

The next step is to find $\beta'_1,\beta'_2,\beta'_3\in\nCSM^2(V)$ such that $\prank(\beta_i-\beta'_i)\leq 48r$. First note that for bilinear forms we have $\arank(\beta_i)=\prank(\beta_i)$. Thus there is a subspace $U\leq V$ of codimension at most $24r$ such that $\beta_i-(\beta_i)_{(12)}$ vanishes on $U\times U$. This is the only condition we need to imply that $\beta_i|_U$ is an nCSM. Extending $\beta_i|_U$ to $\beta_i'\in\nCSM^2(V)$ arbitrarily, we have $\prank(\beta_i-\beta_i')\leq 48r$ by \cref{thm:close-in-rank}. Finally, by \cref{thm:integrate-csm}, there exist non-classical quadratic polynomials $Q_i\in\Poly_{\leqslant 2}(V\to\R/\Z)$ such that $d^2Q_i=\beta'_i$.

We repeat the same averaging and Fourier analysis argument from above to conclude that
\[\abs{\E_{x,h_1,h_2,h_3}b''(h_1,h_2,h_3)\omega^{\beta'_1(h_2,h_3)+\beta'_2(h_1,h_3)+\beta'_3(h_1,h_2)}(\partial_{h_1}\partial_{h_2}\partial_{h_3}g)(x)}\geq \epsilon p^{-290r},\] where $b''(h_1,h_2,h_3)$ is a product just of linear phase functions, say $b''(h_1,h_2,h_3)=\omega^{L_1(h_1)+L_2(h_2)+L_3(h_3)}$.

Now define
\begin{align*}
g_{000}&=f\omega^{P},\\
g_{001}&=f\omega^{P+Q_1},\\
g_{010}&=f\omega^{P+Q_2},\\
g_{011}&=f\omega^{P+Q_1+Q_2+L_3},\\
g_{100}&=f\omega^{P+Q_3},\\
g_{101}&=f\omega^{P+Q_1+Q_3+L_2},\\
g_{110}&=f\omega^{P+Q_2+Q_3+L_1},\\
g_{111}&=f\omega^{P+Q_1+Q_2+Q_3+L_1+L_2+L_3}.
\end{align*}

One can verify that the inequality above can be rewritten as
\[\begin{split}
\left|\E_{x,h_1,h_2,h_3}\ol{g_{000}(x)}\right.&{g_{001}(x+h_1)g_{010}(x+h_2)}\ol{g_{011}(x+h_1+h_2)}{g_{100}(x+h_3)}\\
&\left.\cdot \ol{g_{101}(x+h_1+h_3)g_{110}(x+h_2+h_3)}{g_{111}(x+h_1+h_2+h_3)}\right|\geq\epsilon p^{-290r}.
\end{split}\]
The left-hand side is in the correct form to apply the  Gowers-Cauchy-Schwarz inequality \cite[Lemma B.1.(iv)]{TZ12}, which implies that $\|f\omega^P\|_{U^3}\geq\epsilon p^{-290r}\geq p^{-O_p(\log_p(1/\epsilon)^{O(1)})}$. Finishing with a single application of the $U^3$-inverse theorem, we prove the desired result.
\end{proof}

\section{Concluding remarks}
\label{sec:concl}

The symmetrization results in this paper are very specific to trilinear forms. We pose the following conjecture, which seems surprisingly difficult and may be of independent interest. 

\begin{conj}
\label{thm:sym-conj}
Let $T\colon V^k\to \F_p$ be a $k$-linear form. Suppose that $T$ is close to symmetric in the sense that $\prank(T-T_\pi)\leq r$ for all permutations $\pi\in\mathfrak S_k$. Then there exists a symmetric $k$-linear form $S\colon V^k\to\F_p$ such that $\prank(S-T)\leq r'$ where $r'$ is a polynomial function of $r$.
\end{conj}

Note that this conjecture is only interesting if $p\leq k$, since for $p>k$, the symmetric form $S=\frac1{k!}\sum_\pi T_{\pi}$ can easily be seen to satisfy $\prank(S-T)\leq k!r$.

Combined with the recent work of Gowers and Mili\'cevi\'c \cite{GM20}, \cref{thm:sym-conj} is the missing ingredient which would suffice to give quantitative bounds on the $U^{p+1}$-inverse theorem. For the $U^{k+1}$-inverse theorem with $k>p$ there is additional symmetrization necessary to produce a non-classical symmetric multilinear form from a symmetric form.

\begin{conj}
\label{thm:full-sym-conj}
Let $T\colon V^k\to \F_p$ be a $k$-linear form. Suppose that $T$ satisfies
\[\abs{\E_{x_1,\ldots,x_k}\prod_{I\subseteq[k]}b_I\paren{\sum_{i\in I}x_i}\omega^{T(x_1,\ldots,x_k)}}\geq\epsilon\] for some 1-bounded functions $b_I\colon V\to\C$. Then there exists $S\in\nCSM^k(V)$ such that $\prank(S-T)\leq r'$ where $r'$ is a polynomial function of $\log(1/\epsilon)$.
\end{conj}

Combined with the work of Gowers and Mili\'cevi\'c and this paper's integration result, \cref{thm:integrate-ncsm}, this result would give quantitative bounds on the $U^{k+1}$-inverse theorem over $\F_p^n$ for all $k,p$.

\begin{rem}
After the first version of this paper appeared on arXiv, Mili\'cevi\'c gave a counterexample disproving \cref{thm:sym-conj} in the case of $k=4$ and $p=2$ \cite{Mil21}. It remains an interesting question to determine for which pairs $(k,p)$ \cref{thm:sym-conj} holds as well as to determine the validity of \cref{thm:full-sym-conj}.
\end{rem}

\section*{Acknowledgments} 
The author thanks Aaron Berger for many helpful discussions.

\bibliographystyle{amsplain}


\begin{dajauthors}
\begin{authorinfo}[jtidor]
  Jonathan Tidor\\
  Massachusetts Institute of Technology,\\
  Cambridge, MA 02139, USA\\
  jtidor\imageat{}mit\imagedot{}edu
\end{authorinfo}
\end{dajauthors}

\end{document}